\documentclass[11pt,letterpaper]{article}
\usepackage[utf8]{inputenc}
\usepackage[english]{babel}
\usepackage{amsmath}
\usepackage{amsfonts}
\usepackage{amssymb}
\usepackage{graphicx}
\usepackage{upref,amsthm,amsxtra,exscale}
\usepackage{cite}
\usepackage[colorlinks=true,urlcolor=blue,
citecolor=red,linkcolor=blue,linktocpage,pdfpagelabels,
bookmarksnumbered,bookmarksopen]{hyperref}

\newcommand{\R}{\mathbb{R}}
\newtheorem{theorem}{Theorem}[section]
\newtheorem{corollary}[theorem]{Corollary}
\newtheorem{remark}[theorem]{Remark}
\newtheorem{lemma}[theorem]{Lemma}
\newtheorem{proposition}[theorem]{Proposition}
\numberwithin{equation}{section}

\author{Mónica Clapp\footnote{M. Clapp was supported by UNAM-DGAPA-PAPIIT grant IN100718 (Mexico).}, \quad Liliane A. Maia\footnote{L. Maia was supported by FAPDF 0193.001300/2016, CNPq/PQ 308378/2017-2 (Brazil), and PROEX/CAPES (Brazil).} \quad and \quad Benedetta Pellacci\footnote{B. Pellacci was supported by ``Gruppo Nazionale per l'Analisi Matematica, la Probabilità e le loro Applicazioni'' (GNAMPA) of the Istituto Nazionale di Alta Matematica (INdAM); PRIN-2015KB9WPT Grant: ``Variational methods, with applications to problems in mathematical physics and geometry'',}}

\title{Positive multipeak solutions to a zero mass problem in exterior domains}
\date{\today}

\begin{document}

\maketitle

\begin{abstract}
We establish the existence of positive multipeak solutions to the nonlinear scalar field equation with zero mass
$$-\Delta u = f(u), \qquad u\in D_0^{1,2}(\Omega_R),$$
where $\Omega_R:=\{x \in \mathbb{R}^N:|u|>R\}$ with $R>0$, $N\geq4$, and the nonlinearity $f$ is subcritical at infinity and supercritical near the origin. We show that the number of positive multipeak solutions becomes arbitrarily large as $R \to \infty$. \medskip


\end{abstract}

\section{Introduction}
\label{sec:introduction}

This paper is concerned with the existence of multiple positive solutions to the nonlinear scalar field equation
\begin{equation}
\label{prob}
-\Delta u = f(u), \qquad u\in D_0^{1,2}(\Omega_R),
\end{equation}
in the exterior domains $\Omega_R:=\{x \in \mathbb{R}^N:|u|>R\}$ with $R>0$ and $N\geq4$. The nonlinearity $f$ is assumed to be subcritical at infinity and supercritical near the origin. Our precise assumptions on $f$ are given below. They include the model nonlinearity
\begin{equation} \label{eq:model}
f(s)=\frac{s^{q-1}}{1+s^{q-p}},
\end{equation}
with $2<p<2^*:=\frac{2N}{N-2}<q$.

In their seminal paper \cite{bl}, Berestycki and Lions considered the zero mass problem
\begin{equation}
\label{limprob}
-\Delta u = f(u), \qquad u\in D^{1,2}(\mathbb{R}^N),
\end{equation}
in the whole space $\mathbb{R}^N$. They showed that, if $f$ is subcritical at infinity and supercritical near the origin, it has a ground state solution.

The problem 
\begin{equation}
\label{eq:exterior}
-\Delta u = f(u), \qquad u\in D_0^{1,2}(\Omega),
\end{equation}
in an arbitrary exterior domain $\Omega$ (i.e., in a smooth domain whose complement is bounded and nonempty) was studied by Benci and Micheletti for domains whose complement has small enough diameter \cite{bm}, and by Khatib and Maia in the general case  \cite{km}. They showed that \eqref{eq:exterior} does not have a least energy solution, but that it does have a positive higher energy solution whenever the limit problem \eqref{limprob} has a unique positive solution, up to translations. This is true, e.g., for the model nonlinearity \eqref{eq:model}. 

When $\Omega$ is the complement of a ball, the problem \eqref{prob} is known to have a positive radial solution; see \cite{el}. Thus, it is natural to ask whether the solution found in \cite{km} coincides with the radial one or not. We shall see that it does not, if $R$ is sufficiently large. Moreover, we will show that the number of positive nonradial solutions to the problem \eqref{prob} becomes arbitrarily large, as $R\to\infty$.

We assume that $f$ has the following properties:

\begin{enumerate}
\item[$(f1)$] $f\in\mathcal{C}^{1}[0,\infty),$ and there are
constants $A_{1}>0$ and $2<p<2^{\ast}<q$ such that, for $m=-1,0,1,$
\begin{equation}
|f^{(m)}(s)|\leq
\begin{cases}
A_{1}|s|^{p-(m+1)} & \text{if}\ |s|\geq 1,\\
A_{1}|s|^{q-(m+1)} & \text{if}\ |s|\leq 1,
\end{cases}
\end{equation}
where $f^{(-1)}:=F,$ $f^{(0)}:=f,$ $f^{(1)}:=f',$ and $F(s):=\int_{0}^{s}f(t)\mathrm{d}t.$

\item[$(f2)$] There is a constant $\theta>2$ such that
$0\leq\theta F(s)\leq f(s)s<f'(s)s^{2}$ for all $s>0.$
\end{enumerate}

We write
$$\|u\|^{2}:=\int_{\mathbb{R}^{N}}|\nabla u|^{2}$$
for the norm of $u$ in the space $D^{1,2}(\mathbb{R}^{N})$, and 
$D_{0}^{1,2}(\Omega_{R})$ for the closure of $\mathcal{C}_c^\infty(\Omega_R)$ in $D^{1,2}(\mathbb{R}^{N})$. Let
$$J(u):= \frac{1}{2}\|u\|^{2} - \int_{\mathbb{R}^N} F(u)$$
be the energy functional associated to problem \eqref{prob}.

We identify $\mathbb{R}^{N}\equiv\mathbb{C} \times \mathbb{R}^{N-2}$ and we write the
points in $\mathbb{R}^{N}$ as $x = (z,y)$ with $z\in\mathbb{C}$ and $y\in \mathbb{R}^{N-2}$. 

We will prove the following result.

\begin{theorem}
\label{thm:main}
If $f$ satisfies $(f1)-(f2)$ then, for any given $m\in\mathbb{N}$, $m\geq 2$, there exists $R(m)>0$ such that, for each $R>R(m)$, the problem $(\ref{prob})$ has $m-1$ positive nonradial solutions $u_{R,2},\ldots,u_{R,m}\in D_{0}^{1,2}(\Omega_{R})$ with the following properties: for every $n = 2,\ldots,m$,
\begin{enumerate}
\item[$(a)$]$u_{R,n}(\mathrm{e}^{2\pi\mathrm{i}j/n}z,y) = u_{R,n}(z,y)$ for all $(z,y)\in
\Omega_{R}$ and $j=0,\ldots,n-1$,
\item[$(b)$]$u_{R,n}(z,y_1) = u_{R,n}(z,y_2)$ if $|y_1| = |y_2|$,
\item[$(c)$]$(n-1)c_0<J(u_{R,n})<nc_0$, where $c_0$ is the ground state energy of the limit problem \eqref{limprob}.
\end{enumerate}
Moreover, there are sequences $R_{k}>0$, $\xi_k=(\zeta_{k},0)\in\Omega_{R_{k}}$ and a positive least energy radial solution $\omega$ to the limit problem $(\ref{limprob})$ such that
$$\mathrm{dist}(\xi_k,\partial \Omega_{R_k}) \to \infty \qquad\text{as }\,k\to\infty,$$
$$\lim_{k \to \infty} \left\|u_{R_k,n} - \sum\limits_{j=0}^{n-1} \omega\left(\,\cdot\,- (\mathrm{e}^{2\pi\mathrm{i}j/n}\zeta_k,0) \right)\right\| =0,$$
and 
$$J(u_{R_k,n}) \to nc_0 \qquad\text{as }\,k\to\infty.$$
\end{theorem} \vspace{6pt}

The solution $u_{R,n}$ is obtained by minimizing the energy functional $J$ on the Nehari manifold of functions which have the symmetries described in $(a)$ and $(b)$. Due to the lack of compactness of the funcional $J$, the existence of these minimizers is not obvious. We prove a splitting lemma for the varying domains $\Omega_R$ (see Lemma \ref{lem:splitting}) which yields a condition for the existence of $u_{R,n}$. Fine estimates allow us to show that this condition is satisfied and, thus, to prove that a minimizer exists for every $R>0$ and every $2\leq n \leq \infty$; see Theorem \ref{thm:existence} below.

The splitting lemma for the varying domains $\Omega_R$ allows us also to show that
$$\lim_{R\to\infty}J(u_{R,n}) = nc_0 \quad\text{for each }2 \leq n \leq \infty,$$
and that the limit profile of $u_{R,n}$, as $R\to\infty$, is the one described in Theorem \ref{thm:main}.

The symmetries given by $(a)$ and $(b)$ were used by Li in \cite{ly}
 to obtain multiple positive solutions to a subcritical problem in expanding annuli.  As pointed out by Byeon in \cite{b}, the argument given in \cite{ly} does not carry over to dimension $3$. The same thing happens in our situation: the energy bounds for the minimizers, when $N=3$, do not allow us to distinguish them apart; see Theorem \ref{thm:N=3} below. We believe that Theorem \ref{thm:main} is true also in dimension $3$, but the proof requires a different argument. In fact, Cao and Noussair obtained a similar result in \cite{cn} for a semilinear elliptic equation with positive mass and subcritical nonlinearity, which includes dimension $N=3$.

We wish to stress that, as our solutions are obtained by minimization in a suitable symmetric setting, in contrast to the situation considered in \cite{cm,km}, we do not need any special properties of the positive solution to the limit problem \eqref{limprob}, such as uniqueness or nondegeneracy.

This paper is organized as follows: in Section \ref{sec:limprob} we introduce the variational setting for problem \eqref{prob}. Section \ref{sec:concentration} is devoted to the proof of a splitting lemma for the varying domains $\Omega_R$ in a symmetric setting. In Section \ref{sec:upperbound} we obtain an upper bound for the energy of the symmetric minimizers, that will allow us to derive their existence. Section \ref{sec:multiplicity} contains the proof of our main result. Finally, in Section \ref{sec:N=3} we briefly discuss the $3$-dimensional case.


\section{The limit problem and the variational setting} \label{sec:limprob}

For $s<0$ we define $f(s):=-f(-s)$. Then $f\in\mathcal{C}^1(\mathbb{R})$. Note that, if $u$ is a positive solution of the problem (\ref{prob}) for this new function, it is also a solution of (\ref{prob}) for the original function $f$. Hereafter, $f$ will denote this extension. We will assume throughout that $N\geq 4$ and that $f$ satisfies $(f1)-(f2)$.

Let $D^{1,2}(\mathbb{R}^{N}):=\{u\in L^{2^*}(\mathbb{R}^{N}):\nabla u\in L^{2}(\mathbb{R}^{N},\mathbb{R}^{N})\}$, with its standard scalar product and norm
\[
\langle u,v\rangle :=\int_{\mathbb{R}^{N}}\nabla u\cdot\nabla v,\text{\qquad}\|u\| :=\left(  \int_{\mathbb{R}^{N}}|\nabla u|^{2}\right)^{1/2}.
\]

Since $f\in\mathcal{C}^{1}(\mathbb{R})$ and $f$ satisfies $(f1)$, the limit problem \eqref{limprob} has a ground state solution $\omega\in\mathcal{C}^{2}(\mathbb{R}^{N})$, which is positive, radially symmetric and decreasing in the radial direction; see \cite[Theorem 4]{bl}.

Assumption $(f1)$ implies that $|f(s)|\leq A_{1}|s|^{2^*-1}$ and $|f'(s)|\leq A_{1}|s|^{2^*-2}$, and assumption $(f2)$ yields that $f(s)>0$ if $s>0.$ Therefore, every positive solution $u$ to \eqref{limprob} satisfies the decay estimates
\begin{align} 
A_{2}(1+|x|)^{-(N-2)}\leq u(x) \leq A_{3}(1+|x|)^{-(N-2)}, \label{eq:decay} \\
|\nabla u(x)|\leq A_{3}(1+|x|)^{-(N-1)}, \nonumber
\end{align}
for some positive constants $A_{2}$ and $A_{3}$, and $u$ is radially symmetric and strictly radially decreasing about some point in $\mathbb{R}^{N}$; see Theorem 1.1 and Corollary 1.2 in \cite{v}.

Let $2<p<2^{\ast}<q$. The following proposition, combined with assumption
$(f1),$ provides the interpolation and boundedness properties that are needed
to obtain a good variational setting.

\begin{proposition}
\label{prop:bpr}Let $\alpha,\beta>0$ and $h\in\mathcal{C}^{0}(\mathbb{R})$.
Assume that $\frac{\alpha}{\beta}\leq\frac{p}{q},$ $\beta\leq q,$ and there
exists $M>0$ such that
\[
\left\vert h(s)\right\vert \leq M\min\{\left\vert s\right\vert ^{\alpha
},\left\vert s\right\vert ^{\beta}\}\text{\qquad for every }s\in
\mathbb{R}\text{.}%
\]
Then, for every $t\in\left[  \frac{q}{\beta},\frac{p}{\alpha}\right]  ,$ the
map $D^{1,2}(\mathbb{R}^{N})\rightarrow L^{t}(\mathbb{R}^{N})$ given by
$u\mapsto h(u)$ is well defined, continuous and bounded.
\end{proposition}

\begin{proof}
See \cite[Proposition 3.5]{bpr} and \cite[Proposition 3.1]{cm}.
\end{proof}

Let $F(u):=\int_{0}^{u}f(s)\,\mathrm{d}s\text{.}$ As $\left\vert F(s)\right\vert \leq A_{1}\left\vert s\right\vert ^{2^{\ast}%
}$ and $\left\vert f(s)\right\vert \leq A_{1}\left\vert s\right\vert
^{2^{\ast}-1}$ by assumption $(f1)$, the functionals $\Phi$, $\Psi:D^{1,2}%
(\mathbb{R}^{N})\rightarrow\mathbb{R}$ given by
\[
\Phi(u):=\int_{\mathbb{R}^{N}}F(u),\qquad\Psi(u):=\int_{\mathbb{R}^{N}}f(u)u
\]
are well defined. Using Proposition \ref{prop:bpr} it is easy to show that
$\Phi$ is of class $\mathcal{C}^{2}$ and $\Psi$ is of class $\mathcal{C}^{1}$;
see \cite[Lemma 2.6]{bm} or \cite[Proposition 3.8]{bpr}. Hence, the
functional $J:D^{1,2}(\mathbb{R}^{N})\to\mathbb{R}$ given by
\[
J(u):=\frac{1}{2}\int_{\mathbb{R}^{N}}|\nabla u|^{2}-\int_{\mathbb{R}^{N}}F(u),
\]
is of class $\mathcal{C}^{2},$ with derivative
\[
J^{\prime}(u)v=\int_{\mathbb{R}^{N}}\nabla u\cdot\nabla v -\int_{\mathbb{R}^{N}}f(u)v,\qquad u,v\in D^{1,2}%
(\mathbb{R}^{N}),
\]
and the functional $D^{1,2}(\mathbb{R}^{N})\rightarrow\mathbb{R}$
defined by%
\[
u\mapsto J^{\prime}(u)u=\int_{\mathbb{R}^{N}}|\nabla u|^{2} -\int_{\mathbb{R}^{N}}f(u)u,
\]
is of class $\mathcal{C}^{1}.$


\section{Symmetries and concentration}
\label{sec:concentration}

For $R>0$, let $\Omega_R:=\{x \in \mathbb{R}^N:|u|>R\}$. As usual, $D^{1,2}_0(\Omega_R)$ denotes the closure of $\mathcal{C}^\infty _c(\Omega_R)$ in $D^{1,2}(\mathbb{R}^N)$.

Let $G$ be a closed subgroup of the group $O(N)$ of linear isometries of $\mathbb{R}^N$. A function $u:\mathbb{R}^N \to \mathbb{R}$ is called $G$-invariant if
$$u(gx)=u(x) \qquad \text{for every } g\in G, \, x\in \mathbb{R}.$$
By the principle of symmetric criticality \cite{p}, the $G$-invariant solutions to the problem $(\ref{prob})$ are the critical points of the functional $J$ restricted to the space 
$$D_0^{1.2}(\Omega_R)^G:=\{u \in D_0^{1,2}(\Omega_R):u \, \text{is }G\text{-invariant}\}.$$
The nontrivial ones belong to the set
$$\mathcal{N}_{R}^G:=\{u\in D_0^{1,2}(\Omega_R)^G:u\neq0,\,J'(u)u=0\}.$$
It is shown in \cite{el} that, under our assumptions on $f$, the problem (\ref{prob}) has a positive radial solution; see also Theorem \ref{thm:existence} below. Therefore, $\mathcal{N}_{R}^G \neq \emptyset$ and, hence,
$$c_R^G:= \inf\limits_{u \in \mathcal{N}_{R}^G}J(u)<\infty.$$
Our assumptions on $f$ also imply that $\mathcal{N}_{R}^G$ is a $\mathcal{C}^1$-submanifold of $D_0^{1,2}(\Omega_R)^G$ and a natural constraint for the functional $J$; see \cite[Lemma 3.2]{cm}.

If $G=\{1\}$ is the trivial group, we write $\mathcal{N}_{R}$ and $c_R$ instead of $\mathcal{N}_{R}^{\{1\}}$ and $c_R^{\{1\}}$. Setting $\Omega_0:=\mathbb{R}^N$ we have that $\mathcal{N}_{0}$ and $c_0$ are the Nehari manifold and the ground state energy of the limit problem (\ref{limprob}), and there exists $\varrho >0$ such that
\begin{equation}
\label{eq:nehari}
\|u\| \geq \varrho \qquad \text{for every }\,u\in \mathcal{N}_0;
\end{equation}
see, e.g., \cite[Lemma 3.2]{cm}.

Note that, if $H$ is a closed subgroup of $G$, then 
\begin{equation}
\label{eq:cG}
\mathcal{N}_{R}^G \subset \mathcal{N}_{R}^H \subset \mathcal{N}_{R} \subset \mathcal{N}_{0} \quad \text{and} \quad c_R^G \geq c_R^H \geq c_R \geq c_0 >0.
\end{equation}
Note also that
\begin{equation}
\label{eq:cR}
\mathcal{N}_{R}^G \subset \mathcal{N}_{S}^G \subset \mathcal{N}_{0}^G  \quad \text{and} \quad c_R^G \geq c_S^G \geq c_0^G=c_0 \quad \text{if } S\leq R.
\end{equation}
Next, we introduce the groups that will play a role in the proof of our main result. 

Let $\mathbb{S}^{1}:=\{\mathrm{e}^{\mathrm{i}\vartheta}:\vartheta\in [0,2\pi)\}$ be the group of unit complex numbers. The proper closed subgroups of $\mathbb{S}^{1}$ are the cyclic groups
$$\mathbb{Z}_n: = \left\{\mathrm{e}^{\frac{2\pi\mathrm{i}j}{n}}:j=0,...,n-1\right\}.$$
For $2 \leq n \leq \infty$, we define
$$\Gamma_{\infty}:= \mathbb{S}^{1} \times O(N-2) \qquad \text{and} \qquad \Gamma_n := \mathbb{Z}_n \times O(N-2).$$
The group $\Gamma_{\infty}$ acts coordenatewise on $\mathbb{R}^N \equiv \mathbb{C} \times \mathbb{R}^{N-2}$, i.e., if $(\alpha,\beta)\in\mathbb{S}^{1} \times O(N-2)$ and $(z,y)\in \mathbb{C} \times \mathbb{R}^{N-2}$, then
$$(\alpha,\beta)(z,y):=(\alpha z,\beta y).$$
As usual, we write $Gx:=\{gx: g \in G\}$ for the $G$-orbit of a point $x \in \mathbb{R}^N$.

Our aim is to prove a splitting lemma for the moving domains $\Omega_R$; see Lemma \ref{lem:splitting} below. We start with the following auxiliary lemmas.

\begin{lemma}
\label{lem:orbits}
Let $2 \leq n \leq \infty$ and $(x_k )$ be a sequence in $\mathbb{R}^N$. After passing to a subsequence, there exists a sequence $(\xi_k)$ in $\mathbb{R}^N$ and a constant $C_0 >0$ such that
$$\mathrm{dist}(\Gamma_n x_k,\xi_k) \leq C_0 \qquad\text{for all }\,k,$$
and one of the following statements holds true:
\begin{itemize}
\item either $\xi_k =0$ for all $k$,
\item or $n < \infty$, $\xi_k = (\zeta_k,0) \in \mathbb{C} \times \mathbb{R}^{N-2}$ and $|\alpha \zeta_k - \tilde{\alpha}\zeta_k| \to \infty$ for every $\alpha, \tilde{\alpha} \in \mathbb{Z}_n$ with $\alpha \neq \tilde{\alpha}$,
\item or, for each $m \in \mathbb{N}$, there exist $\gamma_1,\ldots,\gamma_m \in \Gamma_n$ such that $|\gamma_i \xi_k - \gamma_j \xi_k| \to \infty$ if $i \neq j$.
\end{itemize}
\end{lemma}

\begin{proof}
Write $x_k = (z_k,y_k) \in \mathbb{C} \times \mathbb{R}^{N-2}$. There are four possibilities:

$(i)$ If $(x_k )$ is bounded, we set $\xi_k =0$ for all $k$.

$(ii)$ If $n < \infty$, $(x_k )$ is unbounded and $(y_k )$ is bounded, then $(z_k)$ is unbounded. So, passing to a subsequence, we get that $z_k \neq 0$ and, setting $\xi_k := (z_k,0)$, we have that $\mathrm{dist}(\Gamma_n x_k,\xi_k) = |y_k|$ and 
$$|\alpha z_k - \tilde{\alpha}z_k| \geq |\mathrm{e}^{2\pi\mathrm{i}/n} -1|\,|z_k| \to \infty$$
for every $\alpha, \tilde{\alpha} \in \mathbb{Z}_n$ with $\alpha \neq \tilde{\alpha}$.

$(iii)$ If $n < \infty$ and $(y_k )$ is unbounded then, after passing to a subsequence, we get that $y_k \neq 0$ and 
$$\frac{y_k}{|y_k|} \to y.$$
As $y$ lies on the unit sphere $\mathbb{S}^{N-3} \subset \mathbb{R}^{N-2}$ and  $N \geq 4$, for each $m \in \mathbb{N}$ there exist $\beta_1,\ldots,\beta_m \in O(N-2)$ such that $\beta_i y \neq \beta_j y$ if $i \neq j$. Then, there exist $\delta >0$ and $k_0 \in \mathbb{N}$ such that 
$$\left|\beta_i \frac{y_k}{|y_k|} - \beta_j \frac{y_k}{|y_k|}\right| \geq \delta \qquad \text{if } i\neq j \text{ and }k\geq k_0.$$
Setting $\xi_k := x_k$ and $\gamma_i := (1,\beta_i)$, we have that $\mathrm{dist}(\Gamma_n x_k,\xi_k)=0$ and $|\gamma_i \xi_k - \gamma_j \xi_k| = |\beta_i y_k - \beta_j y_k| \geq \delta|y_k| \to \infty$ if $i \neq j$. 

$(iv)$ If $n = \infty$ and $(x_k)$ is unbounded, passing to a subsequence, we get that $x_k \neq 0$. Since the $\Gamma_{\infty}$-orbit of every $x \neq 0$ is homeomorphic either to $\mathbb{S}^1$ or to $\mathbb{S}^1 \times \mathbb{S}^{N-3}$ or to $\mathbb{S}^{N-3}$, as $N \geq 4$, it is an infinite set. So, setting $\xi_k := x_k$ and arguing as in the previous case, we get that $\mathrm{dist}(\Gamma_{\infty} x_k,\xi_k)=0$ and that, for each $m \in \mathbb{N}$, there exist $
\gamma_1,\ldots,\gamma_m \in \Gamma_{\infty}$ such that $|\gamma_i \xi_k - \gamma_j \xi_k| \geq \delta|x_k| \to \infty$ if $i \neq j$.
\end{proof}

\begin{lemma}
\label{lem:bl}
If $u_{k}\rightharpoonup u$ weakly in $D^{1,2}(\mathbb{R}^{N})$ then, passing to a subsequence,
\begin{enumerate}
\item[$(a)$] $\int_{\mathbb{R}^{N}}|f(u_{k})-f(u)||\varphi| = o(1)$ \ for every $\varphi\in\mathcal{C}
_c^{\infty}(\mathbb{R}^{N})$,
\item[$(b)$] $ \int_{\mathbb{R}^{N}}F(u_{k}) - \int_{\mathbb{R}^{N}}F(u_{k}-u) = \int_{\mathbb{R}^{N}}F(u)+o(1)$,
\item[$(c)$]$\int_{\mathbb{R}^N}f(u_k)u_k - \int_{\mathbb{R}^N}f(u_k - u)[u_k- u] = \int_{\mathbb{R}^N}f(u)u + o(1),$
\item[$(d)$] $f(u_{k})-f(u_{k}-u)\to f(u)$ in $(D^{1,2}(\mathbb{R}^{N}))'$.
\end{enumerate}
\end{lemma}

\begin{proof}
The proof of this lemma relies on Proposition \ref{prop:bpr}. Statements $(a)$, $(b)$ and $(d)$ are proved in \cite[Lemma 3.8]{cm}. The proof of $(c)$ uses a similar argument.
\end{proof}

The following vanishing lemma is crucial for the proof of Lemma \ref{lem:splitting} below. We write $B_{\varrho}(y):=\{x\in\mathbb{R}^{N}:\left\vert x-y\right\vert <\varrho\}.$

\begin{lemma}
\label{lem:lions}If $(u_{k})$ is bounded in $D^{1,2}(\mathbb{R}^{N})$ and
there exists $\varrho>0$ such that%
\[
\lim_{k\rightarrow\infty}\left(  \sup_{y\in\mathbb{R}^{N}}\int_{B_{\varrho}(y)}\left\vert u_{k}\right\vert ^{2}\right)  =0,
\]
then $\lim_{k\rightarrow\infty}\int_{\mathbb{R}^{N}}f(u_{k})u_{k}=0$.
\end{lemma}

\begin{proof}
See \cite[Lemma 3.5]{cm}.
\end{proof}

Next, we prove a splitting lemma for moving domains.

\begin{lemma}
\label{lem:splitting}
Fix $2 \leq n \leq \infty$, and let $(R_k)$ be a nondecreasing sequence of positive numbers and $u_k \in D^{1,2}_0(\Omega_{R_k})^{\Gamma_n}$ be such that
\begin{enumerate}
\item[$(i)$]$u_k\rightharpoonup 0$ weakly but not strongly in $D^{1,2}(\mathbb{R}^{N})$,
\item[$(ii)$]$J(u_k) \to c$,
\item[$(iii)$]$J'(u_k)\varphi_k \to 0$ for every sequence $(\varphi_k)$ such that $\varphi_k \in D_0^{1,2}(\Omega_{R_k})$ and $(\varphi_k)$ is bounded in $D^{1,2}(\mathbb{R}^N)$.
\end{enumerate}
Then, $n<\infty$ and, after passing to a subsequence, there exist $\xi_k=(\zeta_k,0) \in \Omega_{R_k}$ such that $\mathrm{dist}(\xi_k,\partial\Omega_{R_k}) \to \infty$, and a nontrivial $O(N-2)$-invariant solution $v$ to the limit problem \eqref{limprob} such that $v \geq 0$ if $u_k \geq 0$ for all $k$, and
\begin{align*}
\|u_k\|^2 - \|u_k - w_k\|^2 &= n\|v\|^2 + o(1),\\
\int_{\mathbb{R}^N}F(u_k) - \int_{\mathbb{R}^N}F(u_k - w_k) &= n \int_{\mathbb{R}^N} F(v) + o(1),\\
\int_{\mathbb{R}^N}f(u_k)u_k - \int_{\mathbb{R}^N}f(u_k - w_k)[u_k- w_k] &= n \int_{\mathbb{R}^N}f(v)v + o(1), 
\end{align*}
where
$$w_k(z,y):= \sum\limits_{j=0}^{n-1}v(\mathrm{e}^{-2\pi\mathrm{i}j/n}z - \zeta_k,y).$$
Moreover, $c \geq nJ(v) \geq nc_0$.
\end{lemma}

\begin{proof}
For simplicity, we write $\Omega_k:=\Omega_{R_k}$. 

By property $(i)$, after passing to a subsequence, we may assume that $(u_k)$ is bounded and bounded away from $0$ in $D_0^{1,2}(\mathbb{R}^N)$. Then, property $(iii)$ implies that 
$$\|u_k\|^2-\int_{\mathbb{R}^N}f(u_k)u_k = J'(u_k)u_k = o(1),$$
and using assumption $(f1)$ we obtain that
$$0<a_1<\int_{\mathbb{R}^N}f(u_k)u_k \leq A_0|u_k|_{2^*}^{2^*} ,$$
where $|\cdot|_{2^*}$ is the norm in $L^{2^*}(\mathbb{R}^N)$. By the vanishing lemma (Lemma \ref{lem:lions}), there exist $a_2>0$ and a sequence
$(x_k)$ in $\mathbb{R}^{N}$ such that
$$\int_{B_1(x_k)}|u_k|^{2}=\sup_{x\in\mathbb{R}^{N}}\int_{B_1(x)}|u_k|^{2}\geq a_2>0\quad\text{for all }\,k.$$
For the sequence $(x_k)$ we choose $(\xi_k)$ as in Lemma \ref{lem:orbits}. Then, $|\gamma_k x_k - \xi_k| \leq C_0$ for some $\gamma_k \in \Gamma_n$ and, as $u_k$ is $\Gamma_n$-invariant, we obtain that
\begin{equation}
\label{eq:nonvanishing}
\int_{B_{C_0+1}(\xi_k)}|u_k|^2 \geq \int_{B_1(\gamma_k x_k)}|u_k|^2 \geq a_2\quad\text{for all }\,k.
\end{equation}
Since $\mathrm{supp}(u_k)\subset\Omega_k,$ this implies that
\begin{equation}
\label{eq:unbdd}
\mathrm{dist}(\xi_k,\Omega_k)\leq C_0+1.
\end{equation}

Set $v_k(x):=u_k(x + \xi_k).$ Passing to a subsequence, we have that $v_k\rightharpoonup v$ weakly in $D^{1,2}(\mathbb{R}^{N})$, $v_k\to v$ a.e. in $\mathbb{R}^{N}$ 
and $v_k\to v$ strongly in $L_{\mathrm{loc}}^{2}(\mathbb{R}^{N})$. Then, $v \geq 0$ if $u_k \geq 0$ for all $k$. Inequality \eqref{eq:nonvanishing} implies that $v\neq0$ and, as $u_k\rightharpoonup 0$ weakly in $D^{1,2}(\mathbb{R}^{N})$, we deduce that $|\xi_k| \to \infty$.

Define $\widetilde{\Omega}_k:=\{x\in\mathbb{R}^{N}:x+\xi_k\in\Omega_{k}\}$. Note that, if $\varphi\in\mathcal{C}_c^{\infty}(\mathbb{R}^{N})$ and $\mathrm{supp}(\varphi)\subset\widetilde{\Omega}_k$ for $k$ large enough, then $\varphi_k(x):=\varphi(x-\xi_k)$ satisfies $\varphi_k\in\mathcal{C}_c^{\infty}(\Omega_k)$ for $k$ large enough, and property $(iii)$ yields
\begin{align}
J'(v_k)\varphi & = \int_{\mathbb{R}^{N}}\left(\nabla v_k\cdot\nabla\varphi - f(v_k)\varphi\right)
 \nonumber\\
&=\int_{\Omega_k}\left(\nabla u_k\cdot\nabla\varphi_k  - f(u_k)\varphi_k\right) = o(1). \label{eq:solution}
\end{align}

Set
$$d_k:=\mathrm{dist}(\xi_k,\partial\Omega_k),$$
and consider the interior unit normal $\eta_k:=\frac{\xi_k}{|\xi_k|}$ to $\partial\Omega_k$ at the point $\frac{R_k\xi_k}{|\xi_k|}$. Passing to a subsequence, we have that $\eta_k\to\eta$. We claim that the sequence $(d_k)$ is unbounded. 

If the sequence $(R_k)$ is bounded then, as $|\xi_k| \to \infty$, this is immediately true. 

If $R_k \to \infty$, arguing by contradiction, we assume that $(d_k)$ is bounded. Then, after passing to a subsequence, we have that $d_k \to d\in [0,\infty)$. We consider two cases. If a subsequence of $(\xi_k)$ satisfies that $\xi_k\in\overline{\Omega}_k$, we set
$$\mathbb{H}:=\{x\in\mathbb{R}^{N}:\eta\cdot x > -d\}.$$
Since $R_k \to \infty$, every compact subset in $\mathbb{R}^{N}\smallsetminus\overline{\mathbb{H}}$ is contained in $\mathbb{R}^{N}\smallsetminus\widetilde{\Omega}_k$ for $k$ large enough and, as $v_k\equiv 0$ in $\mathbb{R}^{N}\smallsetminus\widetilde{\Omega}_k$, we have that $v\in D_0^{1,2}(\mathbb{H})$. Moreover, every compact subset of $\mathbb{H}$ is contained in $\widetilde{\Omega}_k$ for $k$ large enough. So, if $\varphi\in\mathcal{C}_c^{\infty}(\mathbb{H})$, then $\mathrm{supp}(\varphi)\subset\widetilde{\Omega}_k$ for $k$ large enough. Passing to the limit in equation (\ref{eq:solution}), using Lemma \ref{lem:bl}$(a)$, we conclude that $J'(v)\varphi=0$ for every $\varphi\in\mathcal{C}_c^{\infty}(\mathbb{H})$. It follows that $v$ is a nontrivial solution of
$$-\Delta v = f(v),\quad v\in D_0^{1,2}(\mathbb{H}),$$
contradicting the fact that this problem has only the trivial solution; see \cite{el,v}. Likewise, if a subsequence of $(\xi_k)$ satisfies that $\xi_k\in\mathbb{R}^{N}\smallsetminus\Omega_k$, we set
$$\mathbb{H}:=\{x\in\mathbb{R}^{N}:\eta\cdot x > d\},$$
and a similar argument yields a contradiction.

This proves that $(d_k)$ is unbounded, and the inequality (\ref{eq:unbdd}) implies that $\xi_k\in\Omega_k$ and that every compact subset of $\mathbb{R}^{N}$ is contained in $\widetilde{\Omega}_k$ for $k$ large enough. So, passing to the limit in equation (\ref{eq:solution}), we conclude that $v$ is a nontrivial solution to the limit problem $(\ref{limprob})$.

Let $\gamma_1,\ldots,\gamma_m \in \Gamma_n$ be such that 
$|\gamma_i \xi_k - \gamma_j \xi_k| \to \infty$ if $i \neq j$. Then, for 
each $j\in\{1,\ldots,m\}$,
$$
v_k\circ\gamma_j^{-1}- \sum \limits_{i=j+1}^{m}v\circ\gamma_i^{-1}(\,\cdot\,-\gamma_i\xi_k + \gamma_j\xi_k) \rightharpoonup v\circ\gamma_j^{-1}
$$
weakly in $D^{1,2}(\mathbb{R}^{N})$, where the sum is defined to be $0$ if $j=m$. Hence,
\begin{align*}
&\left\|v_k\circ\gamma_j^{-1}- \sum \limits_{i=j+1}^{m}v\circ\gamma_i^{-1}(\,\cdot\,-\gamma_i\xi_k + \gamma_j\xi_k)\right\|^2\\
&\qquad = \left\|v_k\circ\gamma_j^{-1}- \sum \limits_{i=j}^{m}v\circ\gamma_i^{-1}(\,\cdot\,-\gamma_i\xi_k + \gamma_j\xi_k)\right\|^2 + \left\|v\circ\gamma_j^{-1}\right\|^2 + o(1).
\end{align*}
Performing the change of variable $x + \gamma_j \xi_k = \tilde{x}$, recalling that $v_{k}(x)=u_{k}(x+\xi_{k})$ and taking into account that $u_k$ is $\Gamma_n$-invariant we get
\begin{align*}
&\left\|u_k - \sum \limits_{i=j+1}^{m}v\circ\gamma_i^{-1}(\,\cdot\,-\gamma_i\xi_k)\right\|^2\\
&\qquad =\left\|u_k - \sum \limits_{i=j}^{m}v\circ\gamma_i^{-1}(\,\cdot\,-\gamma_i\xi_k)\right\|
^2 + \left\|v\right\|^2 + o(1),
\end{align*}
and iterating this identity for $j={1,\dots,m}$ we obtain
\begin{equation}
\label{eq:1}
\left\|u_k\right\|^{2}=\left\|u_k - w_k \right\|^2 + m\left\|v\right\|^2 + o(1),
\end{equation}
where
$$w_k := \sum\limits_{i=1}^{m}v\circ\gamma_i^{-1}(\,\cdot\,-\gamma_i \xi_k).$$
Similarly, using statements $(b)$ and $(c)$ of Lemma \ref{lem:bl} we get that
\begin{equation}
\label{eq:2}
\int_{\mathbb{R}^{N}}F(u_{k}) = \int_{\mathbb{R}^{N}}F(u_k - w_k) + m \int_{\mathbb{R}^{N}}F(v) + o(1)
\end{equation}
and
\begin{equation}
\label{eq:3}
\int_{\mathbb{R}^N}f(u_k)u_k - \int_{\mathbb{R}^N}f(u_k - w_k)[u_k- w_k] = m\int_{\mathbb{R}^N}f(v)v + o(1).
\end{equation}
As $v$ solves the problem $(\ref{limprob})$, property $(iii)$ and equations $(\ref{eq:1})$ and $(\ref{eq:3})$ yield
$$\left\|u_k - w_k \right\|^2 = \int_{\mathbb{R}^N}f(u_k - w_k)[u_k- w_k] + o(1).$$
Note that assumption $(f2)$ implies that $f(s)s-2F(s) \geq 0$ for every $s \in \mathbb{R}$. Therefore,
\begin{align*}
&\frac{1}{2}\left\|u_k - w_k \right\|^2 - \int_{\mathbb{R}^{N}}F(u_k - w_k) + o(1)\\
& \qquad = \int_{\mathbb{R}^N} \left(\frac{1}{2}f(u_k - w_k)[u_k- w_k] -F(u_k - w_k)\right) \geq 0,
\end{align*}
and from property $(ii)$ and equations \eqref{eq:1} and \eqref{eq:2} we get that
\begin{equation}
\label{eq:4}
c = \lim_{k \to \infty} J(u_k) \geq mJ(v)\geq mc_{0}.
\end{equation}
This says that $m$ cannot be arbitrarily large. So, as $|\xi_k| \to \infty$, the only possibility left in Lemma \ref{lem:orbits} is that $n < \infty$, $\xi_k = (\zeta_k,0) \in \mathbb{C} \times \mathbb{R}^{N-2}$ and $|\mathrm{e}^{2\pi\mathrm{i}i/n} \zeta_k - \mathrm{e}^{2\pi\mathrm{i}j/n}\zeta_k| \to \infty$ for $i,j \in \{0,\ldots,n-1\}$  with $i \neq j$.

Then, $v_k(z,y)=u_k(z+\zeta_k,y)$. As $u_{k}$ is $O(N-2)$-invariant, we get that $v_{k}$ is $O(N-2)$-invariant and, since $v_k \to v$ a.e. in $\mathbb{R}^N$, we have that $v$ is also $O(N-2)$-invariant. 

Finally, if we take $m:=n$ and $\gamma_i:= \mathrm{e}^{2\pi\mathrm{i}(i-1)/n}$, $i=1,\ldots,n$, the statements $(\ref{eq:1})$, $(\ref{eq:2})$, $(\ref{eq:3})$ and $(\ref{eq:4})$ complete the proof of the lemma.
\end{proof}


\section{An upper bound for the energy of symmetric minimizers}
\label{sec:upperbound}

Fix $2 \leq n < \infty$ and a positive radial ground state solution $\omega$ to the limit problem (\ref{limprob}). Then, $\omega$ satisfies the decay estimates \eqref{eq:decay}.

Set $\xi_j :=(\mathrm{e}^{2\pi\mathrm{i}j/n},0)$ and, for each $\rho >0$, let
$$\omega_{j,\rho}(x):=\omega(x-\rho \xi_j)\qquad\text{and}
\qquad\sigma_{\rho}:=\sum\limits_{j=0}^{n-1} \omega_{j,\rho}.$$
Fix $R>0$ and a radial cut-off function $\psi=\psi_R \in \mathcal{C}^{\infty}_c(\mathbb{R}^N)$ such that
\begin{equation}
\label{def:psi}
 0 \leq \psi(x) \leq 1, \qquad \psi(x)=0 \text{ if }|x| \leq R, \qquad \psi (x)=1 \text{ if }|x| \geq 2R.
 \end{equation} 
 Then, $\psi\sigma_{\rho} \in D_0^{1,2}(\Omega_R)^{\Gamma_n}$. The aim of this section is to prove the following result.

\begin{proposition}
\label{prop:upperbound}
For every $R>0$ and $2 \leq n < \infty$ there exists $\rho_0 >0$ such that, for each $\rho > \rho_0$,
\begin{enumerate}
\item[$(a)$]there is a unique $t_{\rho}=t_{R,n,\rho} \in (0,\infty)$ such that $t_{\rho}\psi\sigma_{\rho} \in \mathcal{N}_{R}^{\Gamma_n}$,
\item[$(b)$]$c_R^{\Gamma_n} \leq J(t_{\rho}\psi\sigma_{\rho}) < nc_0$.
\end{enumerate}
\end{proposition} \vspace{6pt}

We start with some lemmas. 

\begin{lemma} \label{lem:power}
\begin{enumerate}
\item[$(a)$] If $y_{0},y\in\mathbb{R}^{N},$ $y_{0}\neq y,$ and $\alpha$ and
$\beta$ are positive constants such that $\alpha+\beta>N$, then there exists
$C_{1}=C_{1}(\alpha,\beta,\lvert y-y_{0}\rvert)>0$ such that
\[
\int_{\mathbb{R}^{N}}\frac{\mathrm{d}x}{(1+|x-\rho y_{0}|)^{\alpha}%
(1+|x-\rho y|)^{\beta}}\leq C_{1}\rho^{-\mu}
\]
for all $R\geq1,$ where $\mu:=\min\{\alpha,\beta,\alpha+\beta-N\}$.

\item[$(b)$] If $y_{0},y\in\mathbb{R}^{N}\smallsetminus\{0\},$ and $\kappa$ and
$\vartheta$ are positive constants such that $\kappa+2\vartheta>N$, then there
exists $C_{2}=C_{2}(\kappa,\vartheta,\lvert y_{0}\rvert,\lvert y\rvert)>0$ such
that
\[
\int_{\mathbb{R}^{N}}\frac{\mathrm{d}x}{(1+|x|)^{\kappa}(1+|x-\rho y_{0}%
|)^{\vartheta}(1+|x-\rho y|)^{\vartheta}}\leq C_{2}\rho^{-\tau},
\]
for all $R\geq1,$ where $\tau:=\min\{\kappa,2\vartheta,\kappa+2\vartheta-N\}$.
\end{enumerate}
\end{lemma}

\begin{proof}
See \cite[Lemma 4.1]{cm}.
\end{proof}

For $i,j\in\{0,\ldots,n-1\}$, $\rho>0$, define
\begin{align*}
\varepsilon_{i,j,\rho}&:=\int_{\mathbb{R}^N}\nabla \omega_{i,\rho}\cdot\nabla \omega_{j,\rho} = \int_{\mathbb{R}^N}f(\omega_{i,\rho})\omega_{j,\rho},\\
\varepsilon_{\rho}&:=\sum\limits_{i \neq j}\varepsilon_{i,j,\rho}.
\end{align*}

\begin{lemma}
\label{lem:estimates}
\begin{itemize}
\item[$(a)$]There are positive constants $C_1$ and $C_2$ such that, 
for every $i\neq j$ and $\rho$ large enough,
$$C_1\rho^{-(N-2)}\leq \varepsilon_{i,j,\rho} \leq C_2\rho^{-(N-2)}.$$
Hence, $\varepsilon_{\rho}\to 0$ 
as $\rho \to\infty$.
\item[$(b)$]There exists $C_0>0$ such that
$$\int_{\mathbb{R}^N}f(s\omega_{i,\rho})t\omega_{j,\rho}\geq C_0 
\rho^{-(N-2)}$$
for every $s,t\geq\frac{1}{2}$ and $\rho$ large enough.
\item[$(c)$]If $\nu>0$ and $i\neq j$ then, as $\rho\to\infty$,
$$\int_{\mathbb{R}^{N}}(\omega_{i,\rho}\,\omega_{j,\rho})^{1+\frac{\nu}{2}} = o(\varepsilon_{\rho}).$$
\item[$(d)$]For every $r>1$ and every compact subset $K$ of $
\mathbb{R}^N$, we have that
$$\int_K |\sigma_{\rho}|^r = o(\varepsilon_{\rho})\qquad\text{and}
\qquad\int_K |\nabla\sigma_{\rho}|^r = o(\varepsilon_{\rho}).$$
\end{itemize}
\end{lemma}

\begin{proof}
The first inequality in statement $(a)$ is a special case of $(b)$. The second one is proved in \cite[Lemma 4.2]{cm}. Statement $(b)$ is proved in \cite[Lemma 4.3]{cm}.

$(c):$ Lemma \ref{lem:power} with $\alpha=\beta=(1+\frac{\nu}{2})(N-2)$ and \eqref{eq:decay} imply that, for some $\mu>N-2$,
\begin{equation*}
\int_{\mathbb{R}^{N}}(\omega_{i,\rho}\,\omega_{j,\rho})^{1+\frac{\nu}{2}} \leq C\rho^{-\mu}=o(\varepsilon_\rho).
\end{equation*}

$(d):$ Since $K$ is compact, there exists $C_{K}>0$ such that $|x - \rho\xi_i| \leq \rho + C_{K}$ for all $x\in K$. So from the decay estimates \eqref{eq:decay} we obtain that
$$\int_K \omega_{i,\rho}^r \leq C\int_K |x - \rho\xi_i|^{-r(N-2)}\mathrm{d}x \leq C \rho^{-r(N-2)},$$
for $\rho$ large enough. As $r>1$, statement $(a)$ yields
$$\frac{1}{\varepsilon_{\rho}} \int_K \omega_{i,\rho}^r \leq C\rho^{-(r-1)(N-2)} \to 0 \qquad\text{ as }\rho \to \infty.
$$
Therefore, 
$$\int_K \sigma_{\rho}^r \leq C\sum_{i=0}^{n-1}\int_K \omega_{i,\rho}^r = o(\varepsilon_{\rho}),$$
as claimed. The other estimate is obtained similarly.
\end{proof}

\begin{lemma} \label{lem:cutoff}
For every $t \in (0,\infty)$ we have that
\begin{align}
\label{eq:psi1}
\|\psi\sigma_{\rho}\|^2 &= \|\sigma_{\rho}\|^2 + o(\varepsilon_{\rho}),\\ \label{eq:psi2}
\int_{\mathbb{R}^N}F(t\psi\sigma_{\rho}) &= \int_{\mathbb{R}^N}F(t\sigma_{\rho}) + t^{2^*}o(\varepsilon_{\rho}),\\ \label{eq:psi3}
\int_{\mathbb{R}^N}f(t\psi\sigma_{\rho})[\psi\sigma_{\rho}] &= \int_{\mathbb{R}^N}f(t\sigma_{\rho})\sigma_{\rho} + t^{2^*}o(\varepsilon_{\rho}).
\end{align}
\end{lemma}

\begin{proof}
Let $u\in D^{1,2} (\mathbb{R}^N)$. An easy computation shows that
$$\|\psi u\|^2 = \|u\|^2 + \int_{\mathbb{R}^N}(\psi^2-1)|\nabla u|^2 - \int_{\mathbb{R}^N}(\psi\Delta\psi)u^2.$$ 
Setting $u=\sigma_{\rho}$ and applying statement $(d)$ of Lemma \ref{lem:estimates} we obtain \eqref{eq:psi1}. Next, note that
$$\int_{\mathbb{R}^N}\left(F(t\psi u) - F(tu)\right) = \int_{B_{2R}(0)}\left(F(t\psi u) - F(tu)\right).$$
By the mean value theorem and assumption $(f1)$ there exists $s=s(x) \in (0,1)$ such that
\begin{equation*}
|F(t\psi u)-F(tu)| = |f\left([(1-s)\psi +s]tu\right)||(1-\psi)tu|\leq A_0|tu|^{2^*}
\end{equation*}
pointwise. Hence
\begin{align*}
&\int_{B_{2R}(0)}|F(t\psi u) - F(tu)|  \leq A_0\,t^{2^*}\int_{B_{2R}(0)}|u|^{2^*}.
\end{align*}
Setting $u=\sigma_{\rho}$ and applying statement $(d)$ of Lemma \ref{lem:estimates} we obtain \eqref{eq:psi2}. The proof of \eqref{eq:psi3} is obtained in a similar way.
\end{proof}

\begin{lemma}
\label{lem:fhomo}
For every $\tau >1$ there exists $C_{\tau} >0$ such that
$$\left| \int_{\mathbb{R}^N}\left(t f(\omega_{i,\rho}) - f(t\omega_{i,\rho})\right) \omega_{j,\rho} \right| \leq C_{\tau}|t-1|\varepsilon_{i,j,\rho},$$
for all $t\in[0,\tau]$.
\end{lemma}

\begin{proof}
Fix $u\in\mathbb{R}$ and consider the function $h(t):=tf(u)-f(tu)$. Assumption $(f1)$ implies that 
$|h'(t)| \leq  C|u|^{2^*-1}$ for all $t\in [0,\tau]$. So, by the mean value theorem,
$$|tf(u)-f(tu)| = |h(t)-h(1)| \leq C|u|^{2^*-1}|t-1|$$
for all $t\in [0,\tau]$. By assumption $(f1),$ we have that $\left\vert f(s)\right\vert \leq
A_{1}\left\vert s\right\vert ^{2^{\ast}-1}$. On the other hand, from the
estimates \eqref{eq:decay} and Lemma \ref{lem:power}(a) we obtain
\begin{align*}
\int_{\mathbb{R}^{N}}\omega_{i,\rho}^{2^{\ast}-1}\,\omega_{j,\rho}  &  \leq
C\int_{\mathbb{R}^{N}}\frac{\mathrm{d}x}{(1+|x-\rho \xi_i|)^{N+2}(1+|x-\rho\xi_j|)^{N-2}} \\
&  \leq C\,\rho^{-(N-2)}.
\end{align*}
Hence, Lemma \ref{lem:estimates} yields
\begin{align*}
\left|\int_{\mathbb{R}^N}\left(t f(\omega_{i,\rho}) - f(t\omega_{i,\rho})\right) \omega_{j,\rho} \right| &\leq C|t-1|\int_{\mathbb{R}^N}|\omega_{i,\rho}|^{2^*-1}\omega_{j,\rho} \\
&\leq C_{\tau}|t-1|\varepsilon_{i,j,\rho},
\end{align*}
as claimed.
\end{proof}

\begin{lemma}
\label{lem:nehari}
There exists $\rho_0 > 0$ such that, for each $\rho > \rho_0$, there is a unique $t_{\rho} \in (0,\infty)$ which satisfies that $t_{\rho}\psi\sigma_{\rho} \in \mathcal{N}_{R}^{\Gamma_n}$, and there exists $t_0>1$ such that $t_{\rho} \in (0,t_0)$ for every $\rho \geq \rho_0$.
\end{lemma}

\begin{proof}
Assumption $(f2)$ implies that the function $\frac{f(t)}{t}$ is strictly increasing in $(0,\infty)$. Hence, for each positive function $u\in D^{1,2}(\mathbb{R}^N)$, the function
$$\frac{J'(tu)u}{t} = \|u\|^2 - \int_{\mathbb{R}^N}\frac{f(tu)}{tu}u^2$$
is strictly decreasing in $t \in (0,\infty)$. Therefore, if there exists $t_u \in (0,\infty)$ such that $J'(t_u u)u = 0$, this number will be unique. Observe that $J'(tu)u > 0$ for $t$ small enough.

We claim that there exist $\rho_1 > 0$, $t_0 > 1$ and $M_0>0$ such that
\begin{equation}
\label{eq:claimNehari}
\frac{J'(t_0\sigma_{\rho})\sigma_{\rho}}{t_0} < -M_0 \qquad \text{for all }\, \rho \geq \rho_1.
\end{equation}
Indeed, as $\omega \in \mathcal{N}_0$, there exist $M >0$ and $t_0 > 1$ such that $\frac{J'(t_0\omega)\omega}{t_0} \leq -M$. Assumption $(f2)$ implies that $f$ is increasing. Hence,
\begin{align*}
\frac{J'(t_0\sigma_{\rho})\sigma_{\rho}}{t_0} &= \|\sigma_{\rho}\|^2 - \frac{1}{t_0}\int_{\mathbb{R}^N}f(t_0\sigma_{\rho})\sigma_{\rho}\\
&= \sum\limits_{i=0}^{n-1}\|\omega_{i,\rho}\|^2 + \varepsilon_{\rho} - \sum\limits_{i=0}^{n-1}\frac{1}{t_0}\int_{\mathbb{R}^N}f(t_0\sigma_{\rho})\omega_{i,\rho}\\
&= \sum\limits_{i=0}^{n-1}\frac{J'(t_0\omega_{i,\rho})\omega_{i,\rho}}{t_0} + \varepsilon_{\rho} - \sum\limits_{i=0}^{n-1}\frac{1}{t_0}\int_{\mathbb{R}^N}\left(f(t_0\sigma_{\rho}) - f(t_0\omega_{i,\rho})\right)\omega_{i,\rho}\\
&< -Mn +o_{\rho}(1),
\end{align*}
which immediately yields \eqref{eq:claimNehari}. 

Lemma \ref{lem:cutoff} implies that there exists $\rho_0 \geq \rho_1$ such that
$$\frac{J'(t_0\psi\sigma_{\rho})[\psi\sigma_{\rho}]}{t_0} = \frac{J'(t_0\sigma_{\rho})\sigma_{\rho}}{t_0} +o(\varepsilon_{\rho}) < -\frac{M_0}{2} \qquad \text{if }\,\rho > \rho_0.$$
Hence, there is a unique $t_{\rho} \in (0,t_0)$ such that $J'(t_{\rho}\psi\sigma_{\rho})[\psi\sigma_{\rho}] = 0$, as claimed.
\end{proof}

\begin{lemma}
\label{lem:trho1}
$t_{\rho} \to 1$ as $\rho \to \infty$.
\end{lemma}

\begin{proof}
From Lemma \ref{lem:cutoff} we get that
\begin{align*}
J'(\psi\sigma_{\rho})[\psi\sigma_{\rho}] &= \|\sigma_{\rho}\|^2 - \int_{\mathbb{R}^N}f(\sigma_{\rho})\sigma_{\rho} + o(\varepsilon_{\rho})\\
&= \sum\limits_{i=0}^{n-1}J'(\omega_{i,\rho})\omega_{i,\rho} + \varepsilon_{\rho} - \sum\limits_{i=0}^{n-1}\int_{\mathbb{R}^N}\left(f(\sigma_{\rho}) - f(\omega_{i,\rho})\right)\omega_{i,\rho}\\ 
&=\varepsilon_{\rho} - \sum\limits_{i=0}^{n-1}\int_{\mathbb{R}^N}\left(f(\sigma_{\rho}) - f(\omega_{i,\rho})\right)\omega_{i,\rho}.
\end{align*}
By the mean value theorem and assumption $(f1)$ there exists $s=s(x) \in (0,1)$ such that
\begin{align*}
|f(\sigma_{\rho}) - f(\omega_{i,\rho})| &= |f'((1-s)\sigma_{\rho} + s \omega_{i,\rho})||\sigma_{\rho} - \omega_{i,\rho}|\\
&\leq  C|\sigma_{\rho}|^{2^*-2}\sum\limits_{j \neq i}\omega_{j,\rho}
\end{align*}
pointwise. Hence, using Hölder's inequality and the fact that $|\rho(\xi_j-\xi_i)|\to\infty$ as $\rho\to\infty$, we get that
\begin{align*}
&\int_{\mathbb{R}^N}|f(\sigma_{\rho}) - f(\omega_{i,\rho})|\,\omega_{i,\rho} \leq C\sum\limits_{j \neq i}\int_{\mathbb{R}^N}\sigma_{\rho}^{2^*-2}\omega_{i,\rho}\,\omega_{j,\rho}\\ 
&\qquad\leq C\sum\limits_{j \neq i}|\sigma_{\rho}|_{2^*}^{2^*-2}\left(\int_{\mathbb{R}^N}(\omega_{i,\rho}\,\omega_{j,\rho})^{2^*/2}\right)^{2/2^*}\\
&\qquad\leq C\sum_{j\neq i}\left(\int_{\mathbb{R}^N}\omega^{2^*/2}\left(\,\cdot\,-\rho(\xi_j - \xi_i)\right) \omega^{2^*/2}\right)^{2/2^*} = o_{\rho}(1).
\end{align*}
We conclude that $J'(\psi\sigma_{\rho})[\psi\sigma_{\rho}] = o_{\rho}(1)$. This implies that $t_{\rho} \to 1$ as $\rho \to \infty$, as claimed.
\end{proof}

\begin{lemma}
\label{lem:acp}
Given $r>0$, $m\in\mathbb{N}$ and $\nu\in(0,q-2)$, there exists a constant $C_{r,m,\nu}>0$ such that, for any finite set of numbers $a_1,\ldots,a_m \in (0,r]$,
\begin{equation}
\label{eq:acp}
F\left(\sum_{i=1}^m a_i\right) - \sum_{i=1}^m F(a_i) - \sum_{\substack{i,j=1\\ i\neq j}}^m f(a_i)a_j\geq - C_{r,m,\nu} \sum_{\substack{i,j=1\\ i\neq j}}^m(a_i a_j)^{1+\frac{\nu}{2}}.
\end{equation}
\end{lemma}

\begin{proof}
First, we claim that, for any $a,b>0$,
\begin{equation}
\label{eq:f}
f(a+b)\geq f(a)+f(b).
\end{equation}
To prove this inequality, observe that we may assume that $a\geq b$. Note that assumption $(f2)$ implies that $\frac{f(t)}{t}$ is increasing in $(0,\infty)$. Then, using $(f2)$ we obtain that
\begin{align*}
f(a+b)- f(a) - f(b) &= \int_a^{a+b}f'(t)\,\mathrm{d}t - f(b) \geq \int_a^{a+b}\frac{f(t)}{t}\,\mathrm{d}t - f(b)\\
&\geq b\left(\frac{f(a)}{a} - \frac{f(b)}{b}\right) \geq 0,
\end{align*}
as claimed.

Now we prove inequality \eqref{eq:acp} by induction on $m$. Hereafter, $C$ will denote some positive constant, not necessarily the same one, which depends only on $r,m,\nu$. 

By \cite[Lemma 4.8]{cm} we have that, for any $a_1,a_2 \in (0,r]$,
\begin{equation}
\label{eq:ew}
F(a_1+a_2)-F(a_1)-F(a_2)-f(a_1)a_2-f(a_1)a_2 \geq - C|a_1a_2|^{1+\frac{\nu}{2}}.
\end{equation}
Let $m\geq 3$ and assume that the inequality \eqref{eq:acp} is true for $m-1$. Then, using the inequalities \eqref{eq:f} and \eqref{eq:ew} we obtain
\begin{align*}
&F\left(\sum_{i=1}^m a_i\right) \geq F\left(\sum_{i=1}^{m-1} a_i\right) + F(a_m) + f\left(\sum_{i=1}^{m-1} a_i\right)a_m \\ 
&\qquad + f(a_m)\left(\sum_{i=1}^{m-1} a_i\right) - C\left(\sum_{i=1}^{m-1} a_i a_m\right)^{1+\frac{\nu}{2}} \\
&\geq \sum_{i=1}^m F(a_i) + \sum_{\substack{i,j=1\\ i\neq j}}^{m-1} f(a_i)a_j + f\left(\sum_{i=1}^{m-1} a_i\right)a_m + f(a_m)\left(\sum_{i=1}^{m-1} a_i\right) \\
&\qquad - C\sum_{\substack{i,j=1\\ i\neq j}}^{m-1} (a_i a_j)^{1+\frac{\nu}{2}} - C\left(\sum_{i=1}^{m-1} a_i a_m\right)^{1+\frac{\nu}{2}} \\
&\geq \sum_{i=1}^m F(a_i) + \sum_{\substack{i,j=1\\ i\neq j}}^m f(a_i)a_j - C\sum_{\substack{i,j=1\\ i\neq j}}^m (a_i a_j)^{1+\frac{\nu}{2}},
\end{align*}
as claimed.
\end{proof}

\begin{proof}[Proof of Proposition \ref{prop:upperbound}] Statement $(a)$ was proved in Lemma \ref{lem:nehari}. The first inequality in statement $(b)$ follows from $(a)$. Next, we prove the second inequality.

From Lemmas \ref{lem:nehari} and \ref{lem:trho1} we have that $t_{\rho}\in [\frac{1}{2},t_0]$ for large enough $\rho$. So, from Lemmas \ref{lem:cutoff}, \ref{lem:acp}, \ref{lem:estimates} and \ref{lem:fhomo} we get that
\begin{align*}
J(t_{\rho}\psi\sigma_{\rho}) &= J(t_{\rho}\sigma_{\rho}) + o(\varepsilon_{\rho})\\
&=\sum_{j=0}^{n-1}J(t_{\rho}\omega_{j,\rho}) + \frac{t_{\rho}^2}{2}\,\sum\limits_{i \neq j}\int_{\mathbb{R}^N}\nabla \omega_{i,\rho}\cdot\nabla \omega_{j,\rho} \\
&\quad + \sum_{j=0}^{n-1}\int_{\R^{N}}F(t_{\rho}\omega_{j,\rho}) - \int_{\R^{N}}F(t_{\rho}\sigma_{\rho}) + o(\varepsilon_{\rho}) \\
&\leq nc_0 + t_{\rho}^2\sum\limits_{i \neq j} \int_{\mathbb{R}^N}f(\omega_{i,\rho})\omega_{j,\rho} - \frac{t_{\rho}^2}{2}\,\varepsilon_{\rho} \\
&\quad -\sum_{i\neq j}\int_{\mathbb{R}^N}f(t_{\rho}\omega_{i,\rho})t_{\rho}\omega_{j,\rho} + C\sum_{i\neq j}\int_{\mathbb{R}^N}t_{\rho}^{2+\nu}(\omega_{i,\rho} \omega_{j,\rho})^{1+\frac{\nu}{2}} + 
o(\varepsilon_{\rho}) \\
& \leq nc_0 + \sum\limits_{i \neq j}\int_{\mathbb{R}^N}\left(t_{\rho} f(\omega_{i,\rho}) 
-f(t_{\rho}\omega_{i,\rho})\right) t_{\rho}\omega_{j,\rho} - \frac{t_{\rho}^2}{2}\,\varepsilon_{\rho} + o(\varepsilon_{\rho}) \\
& \leq nc_0 + C|t_{\rho} -1|\varepsilon_{\rho} - \frac{t_{\rho}^2}{2}\,\varepsilon_{\rho} + o(\varepsilon_{\rho}).
\end{align*}

Lemma \ref{lem:trho1} says that $t_{\rho}\to 1$ as $\rho\to \infty$. This yields the conclusion.
\end{proof}

\section{Multiple positive solutions}
\label{sec:multiplicity}

This section is devoted to the proof of the Theorem \ref{thm:main}.

Fix $R>0$. We write $\nabla J(u)$ for the gradient of the functional $J: D^{1,2}_0(\Omega_R) \to \mathbb{R}$ at $u$ and we write $\nabla_{\mathcal{N}_R^G} J(u)$ for the orthogonal projection of $\nabla J(u)$ onto the tangent space to the Nehari manifold $\mathcal{N}_R^G$ at the point $u$.

A sequence $(u_k)$ will be called a $(PS)_c$-sequence for $J$ on $\mathcal{N}_R^G$ if 
$$u_k \in \mathcal{N}_R^G, \qquad J(u_k) \to c \qquad\text{and}\qquad \left\|\nabla_{\mathcal{N}_R^G} J(u_k)\right\| \to 0.$$ 

\begin{lemma}
\label{lem:ps}
Every $(PS)_c$-sequence $(u_k)$ for $J$ on $\mathcal{N}_R^G$ contains a subsequence which is bounded in $D^{1,2}(\mathbb{R}^N)$ and satisfies $\left\|\nabla J(u_k)\right\| \to 0$.
\end{lemma}

\begin{proof}
By \cite[Lemma 3.6]{cm}, $(u_k)$ contains a bounded subsequence. The same argument given to prove \cite[Lemma 3.7]{cm} shows that $(u_k)$ contains a subsequence such that $\|\nabla_GJ(u_k)\|\to 0$, where $\nabla_GJ(u_k)$ is the gradient of the functional $J: D^{1,2}_0(\Omega_R)^G \to \mathbb{R}$ at $u_k$.

Observe that, if $u\in D^{1,2}_0(\Omega_R)^G$, then $\nabla J(u)\in D^{1,2}_0(\Omega_R)^G$. Therefore, $\nabla J(u)=\nabla_G J(u)$, and the proof is complete.
\end{proof}

\begin{theorem}
\label{thm:existence}
Fix $R>0$ and let $G$ be either the group $O(N)$ or one of the groups $\Gamma_n$ with $2 \leq n \leq \infty$. Then the problem $(\ref{prob})$ has a positive $G$-invariant solution $u$ such that
$$c_0 < J(u)=c_R^G < nc_0.$$
\end{theorem}

\begin{proof}
Let $u_k \in \mathcal{N}_R^G$ be such that $J(u_k) \to c_R^G$. Ekeland's variational principle for $\mathcal{C}^1$-manifolds \cite[Theorem 2.1]{c}, together with Lemma \ref{lem:ps}, allows us to assume that $(u_k)$ is bounded in $D_0^{1,2}(\Omega_R)$ and that $\left\|\nabla J(u_k)\right\| \to 0$.

Passing to a subsequence, we have that $u_k\rightharpoonup u$ weakly in $D^{1,2}_0(\Omega_R)$. Set $v_k := u_k - u \in D^{1,2}_0(\Omega_R)^G$. Using Lemma \ref{lem:bl} we obtain that $u$ is a solution to the problem (\ref{prob}), and that
$$
J(u_k) = J(v_k) + J(u) + o(1) \qquad\text{and}\qquad \nabla J(u_k) = \nabla J(v_k) + o(1).
$$
We claim that $v_k \to 0$ strongly in $D^{1,2}_0(\Omega_R)$.

Arguing by contradiction, assume it does not. Set $R_k := R$. Note that $D^{1,2}_0(\Omega_R)^{O(N)} \subset D^{1,2}_0(\Omega_R)^{\Gamma_{\infty}}$. Then, up to a subsequence, $(v_k)$ satisfies the assumptions $(i)$, $(ii)$ and $(iii)$ of Lemma \ref{lem:splitting} for some $c\leq c^{G}_{R}$, as  $(f_{2})$ yields $J(u)\geq 0$. Consequently, $G \neq O(N)$ and $G \neq \Gamma_{\infty}$. Moreover, if $G=\Gamma_n$ with $n<\infty$ then 
$$c_R^G \geq \lim_{k \to \infty}J(v_k) \geq nc_0,$$
contradicting Proposition \ref{prop:upperbound}. 

This proves that $u_k \to u$ strongly in $D^{1,2}_0(\Omega_R)$. Hence, $u \in \mathcal{N}_R^G$ and $J(u) = c_R^G$. It is shown in \cite[Lemma 2.9]{km} that $c_0$ is the ground state energy of the problem without symmetries in an exterior domain, and that it is not attained. From this fact, and Proposition \ref{prop:upperbound}, we derive that $c_0 < J(u)=c_R^G < nc_0$.

Finally, since $|u|$ is also a minimizer of $J$ on $\mathcal{N}_R^G$, the problem $(\ref{prob})$ has a positive least energy $G$-invariant solution, as claimed.
\end{proof}

As we mentioned earlier, the existence of a positive least energy radial solution to the problem $(\ref{prob})$ was proved in \cite{el}. Theorem \ref{thm:main} asserts that this is not the only positive solution if $R$ is large enough. The main step in its proof is the following lemma.

\begin{lemma}
\label{lem:lowerbound}
For each $2 \leq n < \infty$,
$$\sup_{R>0}c_R^{\Gamma_n} = nc_0.$$
Moreover, if $u_{R,n} \in \mathcal{N}_R^{\Gamma_n}$ is such that $u_{R,n}>0$ and $J(u_{R,n})=c_R^{\Gamma_n}$, then there exist $R_k \to \infty$, $\xi_k = (\zeta_k,0) \in \Omega_{R_k}$ and a positive least energy radial solution $\omega$ to the limit problem $(\ref{limprob})$ such that
$$\mathrm{dist}(\xi_k,\partial \Omega_{R_k}) \to \infty$$
and
$$\lim_{k \to \infty} \left\|u_{R_k,n} - \sum\limits_{j=0}^{n-1} \omega\left(\,\cdot\,-\mathrm{e}^{2\pi\mathrm{i}j/n}\xi_k \right)\right\| =0,$$
where $\mathrm{e}^{2\pi\mathrm{i}j/n}\xi_k:=(\mathrm{e}^{2\pi\mathrm{i}j/n}\zeta_k,0)$.
\end{lemma}

\begin{proof}
Fix $2 \leq n < \infty$. From the inequalities $(\ref{eq:cR})$ and Proposition \ref{prop:upperbound} we get that
$$0<c_0 \leq c:=\sup_{R>0}c_R^{\Gamma_n} \leq nc_0.$$
Fix a sequence $R_k < R_{k+1}$ with $R_k \to \infty$, and let $u_k := u_{R_k,n} \in \mathcal{N}_{R_k}^{\Gamma_n}$ be such that $u_k>0$ and $J(u_k)=c_{R_k}^{\Gamma_n}$. Then, as stated in \eqref{eq:nehari}, $\|u_k\| \geq \varrho >0$ for all $k$. So $(u_k)$ satisfies the assumptions $(i)$, $(ii)$ and $(iii)$ of Lemma \ref{lem:splitting} and, consequently, there exist $\tilde{\xi}_k=(\tilde{\zeta}_k,0) \in \Omega_{R_k}$ such that $\mathrm{dist}(\tilde{\xi}_k,\partial\Omega_{R_k}) \to \infty$, and a positive $O(N-2)$-invariant solution $v$ to the limit problem $(\ref{limprob})$ such that
$$nc_0 \leq nJ(v) \leq c \leq nc_0.$$
Therefore, $c=nc_0$ and $J(v)=c_0$. By \cite[Corollary 1.2]{v}, $v$ is radially symmetric about some point and, as $v$ is $O(N-2)$-invariant, this point is of the form $(\zeta_{*},0) \in \mathbb{C} \times 
\mathbb{R}^{N-2}$. We set $\omega(z,y):=v(z+\zeta_{*},y)$ and $\zeta_k = \tilde{\zeta}_k + \zeta_{*}$. Then, from Lemma \ref{lem:splitting} we get that
\begin{align}
\label{eq:nullenergy}
J(u_k-w_k) &= \frac{1}{2}\|u_k - w_k\|^2 - \int_{\mathbb{R}^N}F(u_k - w_k) =o(1),\\ \label{eq:nearNehari}
\|u_k - w_k\|^2&=\int_{\mathbb{R}^N}f(u_k - w_k)[u_k- w_k] + o(1),
\end{align}
where
$$w_k(z,y):= \sum\limits_{j=0}^{n-1}v(\mathrm{e}^{-2\pi\mathrm{i}j/n}z - \tilde{\zeta}_k,y) =  \sum\limits_{j=0}^{n-1}\omega(z - \mathrm{e}^{2\pi\mathrm{i}j/n}\zeta_k,y).$$
Next we show that $\|u_k - w_k\|^2 \to 0$.

Set $v_k:= u_k - w_k$. Arguing by contradiction, assume that a subsequence satisfies that $\|v_k\|^2 \geq 2a_1 >0$. Then, from equation (\ref{eq:nearNehari}) and assumption $(f1)$ we obtain 
that
$$0<a_1<\int_{\mathbb{R}^N}f(v_k)v_k \leq A_0|v_k|_{2^*}^{2^*}.$$
By Lemma \ref{lem:lions}, there exist $a_2>0$ and a sequence $(y_k)$ in $\mathbb{R}^{N}$ such that
$$\int_{B_1(y_k)}|v_k|^{2}=\sup_{x\in\mathbb{R}^{N}}\int_{B_1(x)}|
v_k|^{2}\geq a_2>0\quad\text{for all }\,k.$$
Define $\tilde{v}_k(x) := v_k(x + y_k)$. As $(\tilde{v}_k)$ is bounded in $D^{1,2}(\mathbb{R}^N)$, after passing to a subsequence, $\tilde{v}_k \rightharpoonup \tilde{v}$ weakly in $D^{1,2}(\mathbb{R}^N)$, $\tilde{v}_k \to \tilde{v}$ a.e. in $\mathbb{R}^N$ and $\tilde{v}_k \to \tilde{v}$ in $L^2_\mathrm{loc}(\mathbb{R}^N)$. Hence, $\tilde{v} \neq 0$. Assumption $(f2)$ implies that $f(s)s-2F(s) >0$ if $s \neq 0$. So, from equations \eqref{eq:nullenergy} and \eqref{eq:nearNehari} and Fatou's lemma, we obtain that
$$0 = \lim_{k \to \infty} \int_{\mathbb{R}^N}\left(\frac{1}{2}f(v_k)v_k - F(v_k)\right) \geq \int_{\mathbb{R}^N}\left(\frac{1}{2}f(\tilde{v})\tilde{v} - F(\tilde{v})\right) > 0,$$
which is a contradiction. This proves that $\|u_k - w_k\|^2 \to 0$, and finishes the proof of the lemma.
\end{proof}

\begin{proof}
[Proof of Theorem \ref{thm:main}]
Fix $m\in\mathbb{N}$, $m\geq 2$. For each $2 \leq n \leq m$ and $R >0$, Theorem \ref{thm:existence} yields a positive least energy $\Gamma_n$-invariant solution $u_{R,n}$ to the problem \eqref{prob}. By Proposition 
\ref{prop:upperbound} and Lemma \ref{lem:lowerbound}, there exists $\tilde{R}_n >0$ such that
$$(n-1)c_0 < J(u_{R,n}) < nc_0 \qquad\text{for all }\,R > \tilde{R}_n.$$
Setting $R(m):= \max\{\tilde{R}_1,\ldots,\tilde{R}_m\}$ we obtain the first part of the statement. The second part is given by Lemma \ref{lem:lowerbound}.
\end{proof}

Lemma \ref{lem:lowerbound} yields also the 
following result.

\begin{corollary}\label{cor:energyradial}
The ground state energies of the radial and the $\Gamma_{\infty}$-invariant solutions satisfy
$$\sup_{R>0}c_R^{\Gamma_{\infty}} = \sup_{R>0}c_R^{O(N)} = \infty.$$
\end{corollary}

\begin{proof}
Indeed, by Lemma \ref{lem:lowerbound} and the inequalities 
\eqref{eq:cG},
$$nc_0 = \sup_{R>0}c_R^{\Gamma_n} \leq \sup_{R>0}c_R^{\Gamma_{\infty}} \leq \sup_{R>0}c_R^{O(N)},$$
for every $2\leq n < \infty$.
\end{proof}

\begin{remark}
Since the energy of the solution obtained in \emph{\cite{km}} is less that $2c_0$, \emph{Corollary} \ref{cor:energyradial} implies that, for $R$ sufficiently large, that solution is different from the radial one if $N\neq3$.
\end{remark}


\section{Some remarks on the 3-dimensional case}
\label{sec:N=3}

If $N=3$ the situation is quite different. For every $2 \leq n \leq \infty$ there are $\Gamma_n$-orbits in $\Omega_R$ which consist of only two points, namely those of the form $\{(0,y),(0,-y)\}$. Therefore, compactness of $\Gamma_n$-invariant $(PS)_c$-sequences is lost already at the level $2c_0$. Lemma \ref{lem:lowerbound} is no longer true in dimension $3$. In fact, one has the following result.

\begin{theorem}
\label{thm:N=3}
Let $N=3$. Fix $R>0$ and $2 \leq n \leq \infty$. Then, the problem $(\ref{prob})$ has a positive $\Gamma_n$-invariant solution $u_{R,n}$ which satisfies
$$c_0 < J(u_{R,n})=c_R^{\Gamma_n} < 2c_0.$$
\end{theorem}

\begin{proof}
We give a sketch of the proof. For each $\rho >0$, let 
$$\tilde{\sigma}_\rho:= \omega_{\rho} + \omega_{-\rho},$$
where $\omega_r(z,y):=\omega(z,y-r)$ for every $(z,y)\in \mathbb{C}\times\mathbb{R}$ and $r\in\mathbb{R}$. Let $\psi = \psi_R$ be a radial cut-off function as in \eqref{def:psi}. Then, $\psi\tilde{\sigma}_{\rho}$ is $[O(2)\times O(1)]$-invariant and, hence, $\psi\tilde{\sigma}_{\rho}\in D_0^{1.2}(\Omega_R)^{\Gamma_n}$ for every $2 \leq n \leq \infty$.

As in Lemma \ref{lem:nehari} one proves that, for $\rho$ large enough, there exists $t_{\rho}\in (0,\infty)$ such that $t_{\rho}\psi\tilde{\sigma}_{\rho} \in \mathcal{N}_R^{\Gamma_n}$. Therefore, $J(t_{\rho}\psi\tilde{\sigma}_{\rho}) \geq c_R^{\Gamma_n}$. Moreover, the argument given to prove statement $(b)$ of Proposition \ref{prop:upperbound} can be easily adapted to show that $J(t_{\rho}\psi\tilde{\sigma}_{\rho}) < 2c_\infty$. 

One can also adapt the proof of Lemma \ref{lem:splitting} to show that $J:\mathcal{N}_R^{\Gamma_n} \to \mathbb{R}$ satisfies the Palais-Smale condition at every $c<2c_0$. Then, using Ekeland's variational principle, we obtain that $c_R^{\Gamma_n}$ is attained by $J$ at some positive function $u_{R,n} \in \mathcal{N}_R^{\Gamma_n}$ which satisfies
$$c_0 < J(u_{R,n})=c_R^{\Gamma_n} < 2c_0,$$
as claimed.
\end{proof}

\begin{remark}\label{remN3}
\emph{Theorem} \ref{thm:N=3} implies that
$$\sup_{R>0}c_R^{\Gamma_n} \leq 2c_0 \quad\text{for every }2 \leq n < \infty,\quad\text{if }N = 3.$$
This stands in contrast with the statement of \emph{Lemma} \ref{lem:lowerbound}, which says that
$$\sup_{R>0}c_R^{\Gamma_n} = nc_0 \quad\text{for each }2 \leq n < \infty,\quad\text{if }N\geq4.$$
So, if $N=3$, the energy bounds for the minimizers do not allow us to distinguish them apart, as they do when $N \neq 3$.

Moreover, higher energy solutions are not easy to get because compactness is lost at many energy levels, for instance, at the levels $2jc_0$ and $c_R^{\Gamma_n}+2jc_0$ for each $j\in\mathbb{N}$, $j\geq1$.
\end{remark}

 \vspace{15pt}

\begin{flushleft}
\textbf{Mónica Clapp}\\
Instituto de Matemáticas\\
Universidad Nacional Autónoma de México\\
Circuito Exterior, Ciudad Universitaria\\
04510 Coyoacán, CDMX\\
Mexico\\
\texttt{monica.clapp@im.unam.mx} \vspace{10pt}

\textbf{Liliane A. Maia}\\
Departamento de Matemática\\
Universidade de Brasília \\
70910-900 Brasília\\
Brazil\\
\texttt{lilimaia@unb.br} \vspace{10pt}

\textbf{Benedetta Pellacci}\\
Dipartimento di Matematica e Fisica\\
Viale Lincoln 5\\
Universit\`a della Campania ``Luigi Vanvitelli''\\
81100 Caserta\\
Italy\\
\texttt{benedetta.pellacci@unicampania.it}
\end{flushleft}

\end{document}